\documentclass[a4paper,12pt]{article}
\usepackage{latexsym}
\usepackage{amsmath}
\usepackage{amssymb}
\usepackage{enumerate}
\usepackage{bm}

\usepackage{theorem}
\newtheorem{theorem}{Theorem}[section]

\newtheorem{lemma}[theorem]{Lemma}
\newtheorem{corollary}[theorem]{Corollary}

\theorembodyfont{\rmfamily}
\newtheorem{proof}{\textmd{\textit{Proof.}}}

\newtheorem{remark}[theorem]{Remark}

\newtheorem{definition}[theorem]{Definition}

\makeatletter

\@addtoreset{equation}{section}
\makeatother

\newcommand{\qedd}{\hfill \Box}
\newcommand{\ve}{\varepsilon}
\newcommand{\del}{\partial}
\newcommand{\lra}{\longrightarrow}

\newcommand{\N}{\ensuremath{\mathbb{N}}}
\newcommand{\R}{\ensuremath{\mathbb{R}}}

\newcommand{\bS}{\ensuremath{\mathbf{S}}}
\newcommand{\cC}{\ensuremath{\mathcal{C}}}
\newcommand{\cE}{\ensuremath{\mathcal{E}}}
\newcommand{\cL}{\ensuremath{\mathcal{L}}}
\newcommand{\cP}{\ensuremath{\mathcal{P}}}
\newcommand{\cT}{\ensuremath{\mathcal{T}}}
\newcommand{\CD}{\ensuremath{\mathsf{CD}}}

\def\vol{\mathop{\mathrm{vol}}\nolimits}
\def\div{\mathop{\mathrm{div}}\nolimits}

\def\loc{\mathop{\mathrm{loc}}\nolimits}
\def\Ric{\mathop{\mathrm{Ric}}\nolimits}
\def\tr{\mathop{\mathrm{tr}}\nolimits}
\def\Lip{\mathop{\mathrm{Lip}}\nolimits}

\newcommand{\Nabla}{\bm{\nabla}}
\newcommand{\Lap}{\bm{\Delta}}

\title{Bochner-Weitzenb\"ock formula and Li-Yau estimates\\
on Finsler manifolds}
\author{Shin-ichi Ohta\thanks{Department of Mathematics, Kyoto University,
Kyoto 606-8502, Japan ({\sf sohta@math.kyoto-u.ac.jp});
Supported in part by the Grant-in-Aid for Young Scientists (B) 20740036.}
\& Karl-Theodor Sturm\thanks{Institut f\"ur Angewandte Mathematik,
Universit\"at Bonn, Endenicher Allee 60, 53115 Bonn, Germany
({\sf sturm@uni-bonn.de})}}
\date{}
\begin{document}

\maketitle

\begin{abstract}
We prove the Bochner-Weitzenb\"ock formula for the (nonlinear) Laplacian
on general Finsler manifolds and derive Li-Yau type gradient estimates
as well as parabolic Harnack inequalities.
Moreover, we deduce Bakry-\'Emery gradient estimates.
All these estimates depend on lower bounds for the weighted flag Ricci tensor.
\end{abstract}

\section{Introduction}

Geometry and analysis on singular spaces is an important topic of current research.
Besides Alexandrov spaces, Finsler manifolds constitute one of the most relevant classes
of explicit examples of metric measure spaces.
Finsler spaces quite often occur naturally via homogenization
as scaling limits of discrete or Riemannian structures.

A \emph{Finsler manifold} is a smooth manifold $M$ equipped with a norm
-- or, more generally, a Minkowski norm -- 
$F(x,\cdot )$ on each tangent space $T_xM$.
The particular case of a Hilbert norm leads to the important subclasses
of Riemannian manifolds.
In our previous paper \cite{OS1}, we introduced and studied in detail
the \emph{heat flow} on a Finsler manifold. It can equivalently be defined
\begin{itemize}
\item
either as gradient flow in $L^2(M,m)$ for the \emph{energy}
$\frac{1}{2} \int_M F(\Nabla u)^2 \,dm$;
\item
or as gradient flow in the  $L^2$-Wasserstein space $(\cP_2(M),W_2)$
for the \emph{relative entropy} $\int_M u\log u \,dm$.
\end{itemize}
For this \emph{nonlinear} evolution semigroup, we proved
$\cC^{1,\alpha}$-regularity and  $L^p$-contraction estimates
as well as integrated upper Gaussian estimates \`a la Davies (\cite[Sections~3, 4]{OS1}).
\bigskip

Various fine properties of the heat flow are intimately linked with
the \emph{weighted flag Ricci tensor} $\Ric_N$ (with $N \in [\dim M,\infty]$)
introduced in \cite{Oint}.
In \cite{Oint}, the first author proved that the curvature-dimension condition
$\CD(K,N)$ in the sense of Lott-Sturm-Villani for a weighted Finsler space
$(M,F,m)$ is equivalent to a lower bound $K$ for $\Ric_N$.
In particular, the relative entropy -- regarded as a function on the
$L^2$-Wasserstein space -- is $K$-convex if and only if
$\Ric_{\infty}$ is bounded from below by $K$.
In \cite[Theorem~5.6]{OS1}, we proved that lower bounds for the weighted flag Ricci curvature imply
point-wise comparison results \`a la Cheeger-Yau.
Surprisingly enough, however, in \cite{OS2} we  observed that exponential
contraction/expansion bounds in Wasserstein distance
\begin{equation}\label{eq:intr1}
W_2(P_t \mu, P_t \nu)\le e^{Ct} W_2(\mu,\nu)
\end{equation}
hold true if and only if we are in a Riemannian setting.
\bigskip

The main goal of this paper now is to derive the \emph{Bochner-Weitzenb\"ock formula}
(Theorems~\ref{th:pwBW}, \ref{th:intBW})
\begin{equation}\label{eq:intr2}
\Delta^{\Nabla u} \bigg( \frac{F(\Nabla u)^2}{2} \bigg) -D(\Lap u)(\Nabla u)
 =\Ric_{\infty}(\Nabla u) +\| \Nabla^2 u \|_{HS(\Nabla u)}^2.
\end{equation}
Here $\Delta^{\Nabla u}$ denotes the linearization of the Laplacian $\Lap$ on $M$
in direction $\Nabla u$.
The formula \eqref{eq:intr2} a priori holds true only on the set
$\{ x \in M \,|\, \Nabla u(x) \neq 0\}$.
It extends to an identity in distributional sense on entire $M$.

Formulation and proof of \eqref{eq:intr2} -- as well as applications of it --
require lot of care due to the lack of regularity.
For instance, solutions to the heat equation $\Lap u=\del u/\del t$
on $M$ will be $\cC^2$ \emph{only if} $M$ is Riemannian (\cite{OS1}).
As a first striking application we deduce the \emph{Bakry-\'Emery gradient estimate}
(Theorem~\ref{th:BE})
\begin{equation}\label{eq:intr3}
F\big( \Nabla P_tu(x) \big)^2 \le
 e^{-2Kt} P_{0,t}^{\Nabla u} \big( F(\Nabla u)^2 \big)(x).
\end{equation}
Here $P_t$ denotes the (nonlinear) semigroup generated by the Laplacian
$\Lap$, whereas $P^{\Nabla u}_{0,t}$ denotes
the (linear, symmetric) Markov transition operator on $L^2(M,m)$
with time-dependent generator $L_s$, $0<s<t$, obtained by linearization of $\Lap$
in direction $\Nabla u(s,\cdot)$.
An immediate consequence of \eqref{eq:intr3} is the growth bound for
Lipschitz constants
\begin{equation}\label{eq:intr4}
\Lip(P_tu) \le e^{-Kt} \Lip(u).
\end{equation}
Note that in the Riemannian setting (or more abstract linear frameworks),
according to the Kantorovich-Rubinstein duality,
the bounds \eqref{eq:intr1} and \eqref{eq:intr4} are equivalent to each other
(with $C=-K$).
\medskip

Finally, we prove the \emph{Li-Yau gradient estimate} (Theorem~\ref{th:LY})
for positive solutions to the heat equation on Finsler spaces which
in the case $\Ric_N\ge0$ simply reads as
\[ F\big( \Nabla (\log u)(t,x) \big)^2 - \del_t(\log u)(t,x)
 \le \frac{N}{2t}, \]
and deduce as a corollary the famous \emph{Li-Yau type Harnack inequality}
(Theorem~\ref{th:LYH})
\[ u(s,x) \le u(t,y) \cdot \bigg( \frac{t}{s} \bigg)^{N/2}
 \exp\bigg( \frac{d(y,x)^2}{4(t-s)} \bigg) \]
for any  $0<s<t$ and $x,y \in M$.

For \eqref{eq:intr3} and the subsequent estimates, $M$ is required to be compact.
It is an open question whether this compactness assumption can be replaced
by completeness.

\section{Finsler geometry}\label{sc:prel}

We briefly review the fundamentals of Finsler geometry,
for which we refer to \cite{BCS} and \cite{Shlec},
as well as some results from \cite{Oint} and \cite{OS1}.

\subsection{Finsler manifolds}

Let $M$ be a connected, $n$-dimensional $\cC^{\infty}$-manifold without boundary.
Given a local coordinate $(x^i)_{i=1}^n$ on an open set $U \subset M$,
we will always use the coordinate $(x^i,v^j)_{i,j=1}^n$ of $TU$ such that
\[ v=\sum_{j=1}^n v^j \frac{\del}{\del x^j}\Big|_x \in T_xM, \qquad x \in U. \]

\begin{definition}[Finsler structures]\label{df:Fstr}
We say that a nonnegative function $F:TM \lra [0,\infty)$ is
a \emph{$\cC^{\infty}$-Finsler structure} of $M$ if the following three conditions hold:
\begin{enumerate}[(1)]
\item(Regularity)
$F$ is $\cC^{\infty}$ on $TM \setminus 0$,
where $0$ stands for the zero section.

\item(Positive $1$-homogeneity)
It holds $F(cv)=cF(v)$ for all $v \in TM$ and $c \ge 0$.

\item(Strong convexity)
The $n \times n$ matrix
\begin{equation}\label{eq:gij}
\big( g_{ij}(v) \big)_{i,j=1}^n :=
 \bigg( \frac{1}{2}\frac{\del^2 (F^2)}{\del v^i \del v^j}(v) \bigg)_{i,j=1}^n
\end{equation}
is positive-definite for all $v \in T_xM \setminus 0$.
\end{enumerate}
We call such a pair $(M,F)$ a \emph{$\cC^{\infty}$-Finsler manifold}.
\end{definition}

In other words, $F$ provides a Minkowski norm on each tangent space
which varies smoothly in the horizontal direction.
For $x,y \in M$, we define the distance from $x$ to $y$ in a natural way by
\[ d(x,y):=\inf_{\eta} \int_0^1 F\big( \dot{\eta}(t) \big) \,dt, \]
where the infimum is taken over all $\cC^1$-curves $\eta:[0,1] \lra M$
such that $\eta(0)=x$ and $\eta(1)=y$.
Note that $d(y,x) \neq d(x,y)$ may happen since $F$ is only positively homogeneous.
A $\cC^{\infty}$-curve $\eta$ on $M$ is called a \emph{geodesic}
if it is locally minimizing and has a constant speed
(i.e., $F(\dot{\eta})$ is constant).
See \eqref{eq:geod} below for the precise geodesic equation.
For $v \in T_xM$, if there is a geodesic $\eta:[0,1] \lra M$
with $\dot{\eta}(0)=v$, then we define the \emph{exponential map}
by $\exp_x(v):=\eta(1)$.
We say that $(M,F)$ is \emph{forward complete} if the exponential
map is defined on whole $TM$.
Then the Hopf-Rinow theorem ensures that any pair of points
is connected by a minimal geodesic (see \cite[Theorem~6.6.1]{BCS}).

For each $v \in T_xM \setminus 0$, the positive-definite matrix
$(g_{ij}(v))_{i,j=1}^n$ in \eqref{eq:gij} induces
the Riemannian structure $g_v$ of $T_xM$ via
\begin{equation}\label{eq:gv}
g_v\bigg( \sum_{i=1}^n a_i \frac{\del}{\del x^i}\Big|_x,
 \sum_{j=1}^n b_j \frac{\del}{\del x^j}\Big|_x \bigg)
 := \sum_{i,j=1}^n g_{ij}(v) a_i b_j.
\end{equation}
This is regarded as the best Riemannian approximation of $F|_{T_xM}$
in the direction $v$.
In fact, the unit sphere of $g_v$ is tangent to that of $F|_{T_xM}$
at $v/F(v)$ up to the second order.
In particular, we have $g_v(v,v)=F(v)^2$.
The \emph{Cartan tensor}
\[ A_{ijk}(v):=\frac{F(v)}{2} \frac{\del g_{ij}}{\del v^k}(v),
 \qquad v \in TM \setminus 0, \]
is a quantity appearing only in the Finsler context.
Indeed, $A_{ijk}$ vanishes everywhere on $TM \setminus 0$
if and only if $F$ comes from a Riemannian metric.
We will repeatedly use the next useful fact on homogeneous functions.

\begin{theorem}{\rm (see \cite[Theorem~1.2.1]{BCS})}\label{th:Euler}
Suppose that a differentiable function $H:\R^n \setminus 0 \lra \R$ satisfies
$H(cv)=c^r H(v)$ for some $r \in \R$ and all $c>0$ and $v \in \R^n \setminus 0$
$($in other words, $H$ is \emph{positively $r$-homogeneous}$)$.
Then we have
\[ \sum_{i=1}^n \frac{\del H}{\del v^i}(v)v^i=rH(v) \]
for all $v \in \R^n \setminus 0$.
\end{theorem}

Observe that $g_{ij}$ is positively $0$-homogeneous on each $T_xM$, and hence
\begin{equation}\label{eq:Av}
\sum_{i=1}^n A_{ijk}(v)v^i =\sum_{j=1}^n A_{ijk}(v)v^j
 =\sum_{k=1}^n A_{ijk}(v)v^k =0
\end{equation}
for all $v \in TM \setminus 0$ and $i,j,k=1,2,\ldots,n$.
Define the \emph{formal Christoffel symbol}
\[ \gamma^i_{jk}(v):=\frac{1}{2}\sum_{l=1}^n g^{il}(v) \bigg\{
 \frac{\del g_{jl}}{\del x^k}(v) +\frac{\del g_{lk}}{\del x^j}(v)
 -\frac{\del g_{jk}}{\del x^l}(v) \bigg\} \]
for $v \in TM \setminus 0$,
where $(g^{ij})$ stands for the inverse matrix of $(g_{ij})$.
We also introduce the \emph{geodesic spray coefficients}
and the \emph{nonlinear connection}
\[ G^i(v):=\sum_{j,k=1}^n \gamma^i_{jk}(v) v^j v^k, \qquad
 N^i_j(v):=\frac{1}{2} \frac{\del G^i}{\del v^j}(v) \]
for $v \in TM \setminus 0$, and $G^i(0)=N^i_j(0):=0$ by convention.
Note that $G^i$ is positively $2$-homogeneous, so that
Theorem~\ref{th:Euler} implies $\sum_{j=1}^n N^i_j(v) v^j=G^i(v)$.

By using the nonlinear connections $N^i_j$,
the coefficients of the \emph{Chern connection} is given by
\begin{equation}\label{eq:Gamma}
\Gamma^i_{jk}:=\gamma^i_{jk}
 -\sum_{l,m=1}^n \frac{g^{il}}{F}(A_{jlm}N^m_k +A_{lkm}N^m_j -A_{jkm}N^m_l)
\end{equation}
on $TM \setminus 0$.
Namely, the corresponding \emph{covariant derivative} of a vector field
$X=\sum_{i=1}^n X^i (\del/\del x^i)$ by $v \in T_xM$
with \emph{reference vector} $w \in T_xM \setminus 0$ is defined by
\begin{equation}\label{eq:covd}
D_v^w X(x):=\sum_{i,j=1}^n \bigg\{ v^j \frac{\del X^i}{\del x^j}(x)
 +\sum_{k=1}^n \Gamma^i_{jk}(w) v^j X^k(x) \bigg\} \frac{\del}{\del x^i}\Big|_x.
\end{equation}
Then the \emph{geodesic equation} is written as, with the help of \eqref{eq:Av},
\begin{equation}\label{eq:geod}
D_{\dot{\eta}}^{\dot{\eta}} \dot{\eta}(t)
 =\sum_{i=1}^n \big\{ \ddot{\eta}^i(t) +G^i\big( \dot{\eta}(t) \big) \big\}
 \frac{\del}{\del x^i} \Big|_{\eta(t)} =0.
\end{equation}

\subsection{Nonlinear Laplacian and the associated heat flow}

Let us denote by $\cL^*:T^*M \lra TM$ the \emph{Legendre transform}.
Precisely, $\cL^*$ is sending $\alpha \in T_x^*M$ to the unique element $v \in T_xM$
such that $\alpha(v)=F^*(\alpha)^2$ and $F(v)=F^*(\alpha)$,
where $F^*$ stands for the dual norm of $F$.
Note that $\cL^*|_{T^*_xM}$ becomes a linear operator only when $F|_{T_xM}$
is an inner product.
For a differentiable function $u:M \lra \R$, the \emph{gradient vector}
of $u$ at $x$ is defined as the Legendre transform of the derivative of $u$,
$\Nabla u(x):=\cL^*(Du(x)) \in T_xM$.
If $Du(x)=0$, then clearly $\Nabla u(x)=0$.
If $Du(x) \neq 0$, then we can write in coordinates
\[ \Nabla u
 =\sum_{i,j=1}^n g^{ij}(\Nabla u) \frac{\del u}{\del x^j} \frac{\del}{\del x^i}. \]
We must be careful when $Du(x)=0$,
because $g_{ij}(\Nabla u(x))$ is not defined as well as
the Legendre transform $\cL^*$ being only continuous at the zero section.
For later convenience we set in general
\begin{equation}\label{eq:M_V}
M_V:=\{ x \in M \,|\, V(x) \neq 0 \}
\end{equation}
for a vector field $V$ on $M$, and $M_u:=M_{\Nabla u}$.
For a differentiable vector field $V$ on $M$ and $x \in M_V$,
we define $\Nabla V(x) \in T_x^*M \otimes T_xM$
by using the covariant derivative \eqref{eq:covd} as
\begin{equation}\label{eq:DV}
\Nabla V(v):=D^V_v V(x) \in T_xM, \qquad v \in T_xM.
\end{equation}
We also set $\Nabla^2 u(x):=\Nabla(\Nabla u)(x)$
for a twice differentiable function $u:M \lra \R$ and $x \in M_u$.

\begin{lemma}\label{lm:D2u}
$\Nabla^2 u(x)$ is symmetric in the sense that
$g_{\Nabla u}(\Nabla^2 u(v),w)=g_{\Nabla u}(\Nabla^2 u(w),v)$
for all $v,w \in T_xM$.
\end{lemma}

\begin{proof}
This can be checked by hand.
Indeed, choosing a coordinate $\{ (\del/\del x^i)|_x \}_{i=1}^n$
orthonormal with respect to $g_{\Nabla u(x)}$, we see by calculation
\begin{align*}
\Nabla^2 u \bigg( \frac{\del}{\del x^i} \bigg)
&=D^{\Nabla u}_{\del/\del x^i} \bigg( \sum_{j,k=1}^n g^{jk}\big( \Nabla u(x) \big)
 \frac{\del u}{\del x^k}(x) \frac{\del}{\del x^j} \bigg) \\
&= \sum_{j=1}^n \bigg\{ \frac{\del^2 u}{\del x^i \del x^j}(x)
 -\sum_{k=1}^n \Gamma^k_{ij}\big( \Nabla u(x) \big) \frac{\del u}{\del x^k}(x) \bigg\}
 \frac{\del}{\del x^j}.
\end{align*}
$\qedd$
\end{proof}

From now on, we fix an arbitrary positive $\cC^{\infty}$-measure $m$ on $M$
as our base measure.
Define the \emph{divergence} of a differentiable vector field $V$ on $M$
with respect to $m$ by
\[ \div_m V:=\sum_{i=1}^n \bigg( \frac{\del V_i}{\del x^i} +V_i \frac{\del \Phi}{\del x^i} \bigg), \]
where we decompose $m$ in coordinates as $dm=e^{\Phi} \,dx^1 dx^2 \cdots dx^n$.
This can be rewritten (and extended to weakly differentiable vector fields) in the weak form as
\[ \int_M \phi \div_m V \,dm =-\int_M D\phi(V) \,dm \]
for all $\phi \in \cC_c^{\infty}(M)$.
Then we define the distributional \emph{Laplacian} of $u \in H^1_{\loc}(M)$ by
$\Lap u:=\div_m(\Nabla u)$ in the weak sense that
\[ \int_M \phi\Lap u \,dm:=-\int_M D\phi(\Nabla u) \,dm \]
for $\phi \in \cC_c^{\infty}(M)$.
We remark that $H^1_{\loc}(M)$ (and $L^2_{\loc}(M)$ etc.)
is defined solely in terms of the differentiable structure of $M$.

Given a vector field $V$ such that $V \neq 0$ on $M_u$,
we can define the gradient vector and the Laplacian
on the weighted Riemannian manifold $(M,g_V,m)$ by
\[ \nabla^V u:=\left\{ \begin{array}{ll}
 \displaystyle\sum_{i,j=1}^n g^{ij}(V) \frac{\del u}{\del x^j} \frac{\del}{\del x^i} & \textrm{on } M_u, \smallskip\\
 0 & \textrm{on } M \setminus M_u, \end{array}\right.
 \qquad \Delta^V u:=\div_m(\nabla^V u), \]
where the latter should be read in the sense of distributions.
Note that, in the definition of $\Delta^V u$, we used the divergence with respect to $m$
rather than the volume form of $g_V$.
It is not difficult to see $\nabla^{\Nabla u}u=\Nabla u$ and $\Delta^{\Nabla u}u=\Lap u$
(\cite[Lemma~2.4]{OS1}).
We also observe for later use that, given $u, f_1$ and $f_2$,
\begin{equation}\label{eq:f1f2}
Df_2(\nabla^{\Nabla u}f_1) =g_{\Nabla u}(\nabla^{\Nabla u}f_1,\nabla^{\Nabla u}f_2)
 =Df_1(\nabla^{\Nabla u}f_2).
\end{equation}

In \cite{OS1}, we have studied the associated \emph{nonlinear heat equation}
$\del u/\del t=\Lap u$.
The nonlinearity is inherited from the Legendre transform.
Given an open set $\Omega \subset M$, define the \emph{Dirichlet energy}
of $u \in H_{\loc}^1(\Omega)$ by
\[ \cE_{\Omega}(u):=\frac{1}{2}\int_{\Omega} F(\Nabla u)^2 \,dm. \]
We will suppress $\Omega$ if $\Omega=M$, namely $\cE=\cE_M$.
Set
\[ H^1(\Omega):=\{ u \in H^1_{\loc}(\Omega) \cap L^2(\Omega,m)
 \,|\, \cE_{\Omega}(u)<\infty \}, \]
and let $H^1_0(\Omega)$ be the closure of $\cC_c^{\infty}(\Omega)$
with respect to the (Minkowski) norm
$\|u\|_{H^1(\Omega)}:=\|u\|_{L^2(\Omega)}+\cE_{\Omega}(u)^{1/2}$.

\begin{definition}[Global and local solutions]\label{df:hf}
\begin{enumerate}[(1)]
\item
For $T>0$, we say that a function $u$ on $[0,T] \times M$
is a \emph{global solution} to the heat equation if
$u \in L^2([0,T],H^1_0(M)) \cap H^1([0,T],H^{-1}(M))$ and if
\[ \int_M \phi \frac{\del u_t}{\del t} \,dm =-\int_M D\phi(\Nabla u_t) \,dm \]
holds for all $t \in [0,T]$ and $\phi \in \cC_c^{\infty}(M)$, where we set $u_t:=u(t,\cdot)$.

\item
Given an open interval $I \subset \R$ and an open set $\Omega \subset M$,
we say that a function $u$ on $I \times \Omega$ is a \emph{local solution}
to the heat equation on $I \times \Omega$ if $u \in L^2_{\loc}(I \times \Omega)$,
$F(\Nabla u) \in L^2_{\loc}(I \times \Omega)$ and if
\[ \int_I \int_{\Omega} u \frac{\del\phi}{\del t} \,dmdt
 =\int_I \int_{\Omega} D\phi(\Nabla u) \,dmdt \]
holds for all $\phi \in \cC_c^{\infty}(I \times \Omega)$.
\end{enumerate}
\end{definition}

Global solutions can be constructed as the gradient flow of the energy functional
$\cE$ in $L^2(M,m)$.
We summarize the existence and regularity properties established in \cite{OS1}
in the next theorem.

\begin{theorem}{\rm (\cite[Sections~3,4]{OS1})}\label{th:hf}
\begin{enumerate}[{\rm (i)}]
\item For each $u_0 \in H^1_0(M)$ and $T>0$,
there exists a global solution $u$ to the heat equation on $[0,T] \times M$,
and the distributional Laplacian $\Lap u_t$ is absolutely continuous
with respect to $m$ for all $t \in (0,T)$.
If $M$ is compact, then such a global solution is uniquely determined
by its initial datum $u_0$.

\item Given an open set $\Omega \subset M$,
the continuous version of any local solution $u$ to the heat equation on $\Omega$ enjoys
the $H^2_{\loc}$-regularity in $x$ as well as the $\cC^{1,\alpha}$-regularity in both $t$ and $x$.
Furthermore, the distributional time derivative $\del_t u$ lies in
$H^1_{\loc}(\Omega) \cap \cC(\Omega)$.
\end{enumerate}
\end{theorem}

We remark that the mild smoothness assumption \cite[(4.4)]{OS1}
clearly holds true for our $\cC^{\infty}$-smooth $F$ and $m$.

\subsection{Ricci curvature}

The \emph{Ricci curvature} (as the trace of the \emph{flag curvature}) on a Finsler manifold
is defined by using the Chern connection (and is in fact independent of the choice of connection).
Instead of giving a precise definition in coordinates,
here we explain a useful interpretation due to Shen \cite[\S 6.2]{Shlec}.
Given a unit vector $v \in T_xM$ (i.e., $F(v)=1$), we extend it to a $\cC^{\infty}$-vector field $V$
on a neighborhood of $x$ in such a way that every integral curve of $V$ is geodesic,
and consider the Riemannian structure $g_V$ induced from \eqref{eq:gv}.
Then the Ricci curvature $\Ric(v)$ of $v$ with respect to $F$ coincides with
the Ricci curvature of $v$ with respect to $g_V$
(in particular, it is independent of the choice of $V$).

Inspired by the above interpretation of the Ricci curvature,
the \emph{weighted Ricci curvature} for $(M,F,m)$ is introduced in \cite{Oint} as follows.

\begin{definition}[Weighted Ricci curvature]\label{df:wRic}
Given a unit vector $v \in T_xM$, let $\eta:(-\ve,\ve) \lra M$ be the geodesic
such that $\dot{\eta}(0)=v$.
We decompose $m$ as $m=e^{-\Psi}\vol_{\dot{\eta}}$ along $\eta$,
where $\vol_{\dot{\eta}}$ is the volume form of $g_{\dot{\eta}}$.
Define
\begin{enumerate}[(1)]
\item $\Ric_n(v):=\displaystyle \left\{
 \begin{array}{ll} \Ric(v)+(\Psi \circ \eta)''(0) & {\rm if}\ (\Psi \circ \eta)'(0)=0, \\
 -\infty & {\rm otherwise}, \end{array} \right.$

\item $\Ric_N(v):=\Ric(v) +(\Psi \circ \eta)''(0) -\displaystyle\frac{(\Psi \circ \eta)'(0)^2}{N-n}\ $
for $N \in (n,\infty)$,

\item $\Ric_{\infty}(v):=\Ric(v) +(\Psi \circ \eta)''(0)$.
\end{enumerate}
For $c \ge 0$ and $N \in [n,\infty]$, we define $\Ric_N(cv):=c^2 \Ric_N(v)$.
\end{definition}

It is established in \cite[Theorem~1.2]{Oint} that, for $K \in \R$,
the bound $\Ric_N(v) \ge KF(v)^2$  is equivalent to Lott, Villani and the second author's
\emph{curvature-dimension condition} $\CD(K,N)$.
This extends the corresponding result on (weighted) Riemannian manifolds
(due to \cite{vRS}, \cite{Stcon}, \cite{StI}, \cite{StII}, \cite{LVII}, \cite{LVI}),
and has many analytic and geometric applications
(see \cite{Oint}, \cite{OS1} or a survey \cite{Osur}).

\begin{remark}\label{rm:Ric}
(a) In contrast to $\Delta^{\Nabla u}u=\Lap u$, $\Ric_N(\Nabla u)$
may not coincide with the weighted Ricci curvature $\Ric_N^{\Nabla u}(\Nabla u)$
of the weighted Riemannian manifold $(M,g_{\Nabla u},m)$.
Indeed, on a Minkowski space $(\R^n,\|\cdot\|)$, it is easy to see that
$(\R^n,g_{\Nabla u})$ is not flat for some function $u$.

(b) For a Riemannian manifold $(M,g,\vol_g)$ endowed with the Riemannian volume measure,
clearly we have $\Psi \equiv 0$ and hence $\Ric_N =\Ric$ for all $N \in [n,\infty]$.
It is also known that, for Finsler manifolds of \emph{Berwald type}
(i.e., $\Gamma_{ij}^k$ is constant on each $T_xM$),
the \emph{Busemann-Hausdorff measure} satisfies $(\Psi \circ \eta)' \equiv 0$
(in other words, Shen's $\bS$-curvature vanishes, see \cite[\S 7.3]{Shlec}).
In general, however, there may not exist any measure with vanishing
$\bS$-curvature (see \cite{ORand} for such an example).
This is a reason why we begin with an arbitrary measure $m$.
\end{remark}

For later convenience, we introduce the following notations.

\begin{definition}[Reverse Finsler structure]\label{df:rev}
Define the \emph{reverse Finsler structure} $\overleftarrow{F}$ of $F$ by
$\overleftarrow{F}(v):=F(-v)$.
We say that $F$ is \emph{reversible} if $\overleftarrow{F}=F$.
We will put an arrow $\leftarrow$ on those quantities associated with $\overleftarrow{F}$,
for example, $\overleftarrow{d}\!(x,y)=d(y,x)$, $\overleftarrow{\Nabla}u=-\Nabla(-u)$
and $\overleftarrow{\Ric}_N(v)=\Ric_N(-v)$.
\end{definition}

We say that $(M,F)$ is \emph{backward complete} if $(M,\overleftarrow{F})$
is forward complete.
Compact Finsler manifolds are both forward and backward complete.

\section{Bochner-Weitzenb\"ock formula}\label{sc:BW}

In this section, we prove the Bochner-Weitzenb\"ock formula.
For the sake of simplicity, we first derive a general formula for vector fields $V$,
and then apply it to gradient vector fields of functions.
We start with a point-wise calculation on $M_V$ (recall \eqref{eq:M_V})
followed by an integrated formula.

\subsection{Point-wise calculation}

We follow the argumentation in \cite[Chapter 14]{Vi2} in the Riemannian situation,
with the help of \cite{Oint} and \cite{OS1}.

As in the previous section,
let $(M,F)$ be an $n$-dimensional $\cC^{\infty}$-Finsler manifold
and $m$ be a positive $\cC^{\infty}$-measure on $M$.
Fix a $\cC^{\infty}$-vector field $V$ on $M$ and $x \in M_V$.
For $t \in (0,\delta]$ with sufficiently small $\delta>0$,
we introduce the map $\cT_t$ and the vector field $V_t$
on a neighborhood of $x$ by
\begin{equation}\label{eq:V_t}
\cT_t(y):=\exp_y[tV(y)], \qquad
 V_t\big( \cT_t(y) \big) :=\frac{\del \cT_t}{\del t}(y).
\end{equation}
As the curve $\sigma(t):=\cT_t(y)$ is geodesic, taking the covariant differentiation
$D^{\dot{\sigma}}_{\dot{\sigma}}$ of $t \longmapsto V_t(\sigma(t))$ yields the
\emph{pressureless Euler equation}
\begin{equation}\label{eq:Euler}
\frac{\del V_t}{\del t} +D_{V_t}^{V_t} V_t =0,
\end{equation}
where we abbreviated as
\[ \frac{\del V_t}{\del t}
 =\sum_{i=1}^n \frac{\del V^i_t}{\del t}\frac{\del}{\del x^i}, \qquad
 \textrm{with}\ V_t=\sum_{i=1}^n V^i_t \frac{\del}{\del x^i}. \]

Now we put $\eta(t):=\cT_t(x)$, take an orthonormal basis
$\{ e_i \}_{i=1}^n$ of $(T_xM,g_V)$
such that $e_n=\dot{\eta}(0)/F(\dot{\eta}(0))$, and consider the vector field along $\eta$
\[ E_i(t):=D(\cT_t)_x(e_i) \in T_{\eta(t)}M, \qquad i=1,2,\ldots,n. \]
Then each $E_i$ is a Jacobi field along $\eta$.
(Note that $E_n(t) \neq \dot{\eta}(t)/F(\dot{\eta}(t))$ in general,
since $\cT_t(\eta(s)) \neq \eta(s+t)$ for $s,t>0$.)
We set $E'_i(t):=D^{\dot{\eta}}_{\dot{\eta}}E_i(t)$ for simplicity.
Define the $n \times n$ matrix-valued function $B(t)$ by
$E'_i(t) =\sum_{j=1}^n b_{ij}(t)E_j(t)$.
(This $B$ corresponds to $U$ in \cite[Chapter 14]{Vi2}.)
Along the discussion in \cite[Lemma~7.2]{Oint},
we obtain the following \emph{Riccati type equation}.

\begin{lemma}[Riccati type equation]\label{lm:Ricc}
For $\eta$ and $B$ as above, we have
\[ \frac{d(\tr B)}{dt}(t) +\tr\!\big( B(t)^2 \big) +\Ric\!\big( \dot{\eta}(t) \big)=0,
\qquad t \ge 0. \]
\end{lemma}

\begin{proof}
We give only an outline of the proof, see \cite{Oint} for details.
Consider the matrix-valued function $A(t)=(g_{\dot{\eta}}(E_i(t),E_j(t)))$
and observe $A'=BA+AB^T$ by definition ($B^T$ is the transpose of $B$).
Since each $E_i$ is a Jacobi field, we have $(BA-AB^T)' \equiv 0$, so that
\[ BA-AB^T \equiv B(0)-B(0)^T, \qquad A'=2BA-B(0)+B(0)^T. \]
We also deduce from the Jacobi equation for $E_i$ that
\[ A''=-2\Ric_{\dot{\eta}} +2BAB^T
 =-2\Ric_{\dot{\eta}} +2B^2 A -2B\big( B(0)-B(0)^T \big), \]
where
$(\Ric_{\dot{\eta}})_{ij}
:=g_{\dot{\eta}}(R^{\dot{\eta}}(E_i,\dot{\eta})\dot{\eta},E_j)$
and $R^{\dot{\eta}}$ is the curvature tensor.
Comparing this with $A''=2B'A+2BA'$,
we find $B'=-\Ric_{\dot{\eta}} A^{-1} -B^2$.
Taking the trace with respect to $g_{\dot{\eta}}$ completes the proof.
$\qedd$
\end{proof}

\begin{lemma}\label{lm:B(t)}
\begin{enumerate}[{\rm (i)}]
\item
We have $B(t)=\Nabla V_t(\eta(t))$ in the sense that,
for each $i=1,2,\ldots,n$, $\Nabla V_t(E_i(t))=\sum_{j=1}^n b_{ij}(t)E_j(t)$.

\item
It holds that $\tr(B(t))=\div_m V_t(\eta(t))+D\Psi(\dot{\eta}(t))$,
where $m=e^{-\Psi}\vol_{\dot{\eta}}$ along $\eta$ such that $\vol_{\dot{\eta}}$
denotes the Riemannian volume measure of $g_{\dot{\eta}}$.
\end{enumerate}
\end{lemma}

\begin{proof}
(i) By the definition \eqref{eq:DV},
$\Nabla V_t(E_i(t)) =D_{E_i(t)}^{V_t}V_t(\eta(t))$.
Setting
\[ E_i(t) =\frac{\del}{\del \delta}
 \Big[ \cT_t(\exp_x \delta e_i) \Big] \Big|_{\delta=0}
 =: \sum_{k=1}^n E_i^k(t) \frac{\del}{\del x^k}\Big|_{\eta(t)}, \]
we have
\[ E'_i(t) =\sum_{k=1}^n \frac{dE_i^k}{dt}(t) \frac{\del}{\del x^k}
 +\sum_{j,k,l=1}^n \Gamma^l_{jk} \big( \dot{\eta}(t) \big) \dot{\eta}^j(t) E^k_i(t) \frac{\del}{\del x^l}, \]
where $\dot{\eta}(t)=\sum_{j=1}^n \dot{\eta}^j(t) (\del/\del x^j)|_{\eta(t)}$.
Exchanging the order of differentiations (by $\delta$ and $t$) in the first term,
we deduce from \eqref{eq:V_t} that
\[ \frac{dE_i^k}{dt}(t) =\frac{\del}{\del \delta}
 \Big[ V_t^k \big( \cT_t(\exp_x \delta e_i) \big) \Big] \Big|_{\delta=0}. \]
Therefore we obtain
$E'_i(t)=D_{E_i(t)}^{V_t}V_t(\eta(t))=\Nabla V_t(E_i(t))$
and complete the proof.

(ii) Choose a coordinate $(x^i)_{i=1}^n$ around $\eta(t)$ such that
$\{ (\del/\del x^i)|_{\eta(t)} \}_{i=1}^n$ is an orthonormal basis of
$(T_{\eta(t)}M,g_{V_t})$.
We will suppress evaluations at $\eta(t)$.
Recall first that
\[ \Nabla V_t \bigg( \frac{\del}{\del x^i} \bigg)
 =D_{\del/\del x^i}^{V_t}V_t
 =\sum_{k=1}^n \bigg\{ \frac{\del V_t^k}{\del x^i}
 +\sum_{j=1}^n \Gamma^k_{ij}(V_t) V_t^j \bigg\}
 \frac{\del}{\del x^k}. \]
Thus we have
\[ \tr\!\big(B(t) \big)=\tr(\Nabla V_t)
 =\sum_{i=1}^n \bigg\{ \frac{\del V_t^i}{\del x^i}
 +\sum_{j=1}^n \Gamma^i_{ij}(V_t) V_t^j \bigg\}. \]
We also find, by \eqref{eq:Gamma} and \eqref{eq:Av},
\[ \sum_{j=1}^n \Gamma^i_{ij}(V_t) V_t^j
 =\sum_{j=1}^n \bigg\{ \frac{1}{2} \frac{\del g_{ii}}{\del x^j}(V_t) V_t^j
 -A_{iij}(V_t) \frac{G^j(V_t)}{F(V_t)} \bigg\}. \]

Next we observe from $dm=e^{-\Psi}\sqrt{\det g_{\dot{\eta}(t)}} \,dx^1 dx^2 \cdots dx^n$
on $\eta$ that
\[ \div_m V_t \big( \eta(t) \big) +D\Psi\big( \dot{\eta}(t) \big)
 =\sum_{i=1}^n \frac{\del V_t^i}{\del x^i}
 +\frac{d(\log\sqrt{\det g_{\dot{\eta}(t)}})}{dt}. \]
Note that
\begin{align*}
\frac{d(\log\sqrt{\det g_{\dot{\eta}(t)}})}{dt}
&=\frac{1}{2} \tr\bigg( \frac{d g_{\dot{\eta}(t)}}{dt} \cdot g_{\dot{\eta}(t)}^{-1} \bigg) \\
&= \frac{1}{2} \sum_{i,j=1}^n \bigg\{ V_t^j \frac{\del g_{ii}}{\del x^j}(V_t)
 +\frac{\del g_{ii}}{\del v^j}(V_t) \ddot{\eta}^j(t) \bigg\}.
\end{align*}
Since $\eta$ is a geodesic, we obtain from \eqref{eq:geod} that
\[ \frac{d(\log\sqrt{\det g_{\dot{\eta}(t)}})}{dt}
 = \sum_{i,j=1}^n \bigg\{ \frac{1}{2}
 \frac{\del g_{ii}}{\del x^j}(V_t) V_t^j
 -\frac{A_{iij}(V_t)}{F(V_t)} G^j(V_t) \bigg\}
 = \sum_{i,j=1}^n \Gamma^i_{ij}(V_t) V_t^j. \]
This completes the proof.
$\qedd$
\end{proof}

We deduce from Lemmas~\ref{lm:Ricc}, \ref{lm:B(t)}(ii) that
\[ \frac{d}{dt}\Big|_{t=0+} \Big[ \div_m V_t\big( \eta(t) \big) \Big]
 +\tr\big( B(0)^2 \big)
 +\Ric_{\infty}\! \big( \dot{\eta}(0) \big)=0. \]
Thanks to \eqref{eq:Euler}, we have
\begin{align*}
\frac{d}{dt}\Big|_{t=0+} \Big[ \div_m V_t \big( \eta(t) \big) \Big]
&= D(\div_m V)\big( \dot{\eta}(0) \big)
 +\div_m \bigg(\frac{\del V_t}{\del t}\Big|_{t=0+} \bigg)(x) \\
&= D(\div_m V)\big( \dot{\eta}(0) \big) -\div_m (D_V^V V)(x).
\end{align*}
Combining these, we obtain
\begin{equation}\label{eq:V-BW}
\div_m (D_V^V V) -D(\div_m V)(V) =\Ric_{\infty}(V) +\tr\big( B(0)^2 \big)
\end{equation}
at $x$.

Now, we take $u \in \cC^{\infty}(M)$ and $x \in M_u$,
and apply \eqref{eq:V-BW} to $V=\Nabla u$.
Then the symmetry of $\Nabla^2 u$ (Lemma~\ref{lm:D2u}) allows us to simplify
\eqref{eq:V-BW} in two respects.
First, Lemma~\ref{lm:B(t)}(i) immediately yields
$\tr(B(0)^2)=\| \Nabla^2 u(x) \|^2_{HS(\Nabla u)}$,
where $\|\cdot\|_{HS(\Nabla u)}$ stands for the Hilbert-Schmidt norm
with respect to $g_{\Nabla u}$.
Second, for each $i=1,2,\ldots,n$,
\[ g_{\Nabla u}\bigg( D^{\Nabla u}_{\Nabla u} (\Nabla u),\frac{\del}{\del x^i} \bigg)
 =g_{\Nabla u}\big( \Nabla u,D^{\Nabla u}_{\del/\del x^i} (\Nabla u) \big)
 =\frac{\del}{\del x^i} \bigg( \frac{g_{\Nabla u}(\Nabla u,\Nabla u)}{2} \bigg). \]
This implies
\[ D^{\Nabla u}_{\Nabla u} (\Nabla u) =\sum_{i,j=1}^n g^{ji}(\Nabla u)
 g_{\Nabla u}\bigg( D^{\Nabla u}_{\Nabla u} (\Nabla u),\frac{\del}{\del x^i} \bigg) \frac{\del}{\del x^j}
 =\nabla^{\Nabla u} \bigg( \frac{F(\Nabla u)^2}{2} \bigg). \]
Plugging these into \eqref{eq:V-BW}, we obtain \eqref{eq:BW} below on $M_u$.

\begin{theorem}[Point-wise Bochner-Weitzenb\"ock formula]\label{th:pwBW}
Given $u \in \cC^{\infty}(M)$, we have
\begin{equation}\label{eq:BW}
\Delta^{\Nabla u} \bigg( \frac{F(\Nabla u)^2}{2} \bigg) -D(\Lap u)(\Nabla u)
 =\Ric_{\infty}(\Nabla u) +\| \Nabla^2 u \|_{HS(\Nabla u)}^2
\end{equation}
as well as
\begin{equation}\label{eq:N-BW}
\Delta^{\Nabla u} \bigg( \frac{F(\Nabla u)^2}{2} \bigg) -D(\Lap u)(\Nabla u)
 \ge \Ric_N(\Nabla u) +\frac{(\Lap u)^2}{N}
\end{equation}
for $N \in [n,\infty]$, point-wise on $M_u$.
\end{theorem}

\begin{proof}
We have seen that \eqref{eq:BW} holds on $M_u$.
The second assertion is clear if $N=\infty$.
For $N \in (n,\infty)$, we observe from Lemma~\ref{lm:B(t)} that,
as $\{e_i\}_{i=1}^n$ is orthonormal so that $B(0)$ is symmetric,
\begin{align*}
\| \Nabla^2 u \|_{HS(\Nabla u)}^2 &=\tr\big( B(0)^2 \big)
 =\frac{(\tr B(0))^2}{n} +\bigg\| B(0)-\frac{\tr B(0)}{n}I_n \bigg\|_{HS}^2 \\
&\ge \frac{(\Lap u+D\Psi(\Nabla u))^2}{n}.
\end{align*}
Plugging $a=\Lap u$ and $b=D\Psi(\Nabla u)$ into
\[ \frac{(a+b)^2}{n}
 =\frac{a^2}{N} -\frac{b^2}{N-n} +\frac{N(N-n)}{n}\bigg( \frac{a}{N}+\frac{b}{N-n} \bigg)^2, \]
we obtain
\[ \frac{(\Lap u+D\Psi(\Nabla u))^2}{n}
 \ge \frac{(\Lap u)^2}{N} -\frac{D\Psi(\Nabla u)^2}{N-n}. \]
Together with \eqref{eq:BW}, this yields the desired estimate \eqref{eq:N-BW}.
The remaining case of $N=n$ is derived as the limit.
$\qedd$
\end{proof}

\begin{remark}\label{rm:BW}
Our Bochner-Weitzenb\"ock formula \eqref{eq:BW} can not be simply derived
from the formula of the weighted Riemannian manifold $(M,g_{\Nabla u},m)$.
This is because $\Ric_{\infty}(\Nabla u)$ and $\| \Nabla^2 u \|_{HS(\Nabla u)}^2$
are different from those for $g_{\Nabla u}$ (unless all integral curves of $\Nabla u$ are geodesic),
while $D(\Lap u)(\Nabla u)$ and $\Delta^{\Nabla u}(F(\Nabla u)^2/2)$ are same for $F$ and $g_{\Nabla u}$.
Recall Remark~\ref{rm:Ric}(a) and note that $\Nabla^2 u(v) \neq (\nabla^{\Nabla u})^2 u(v)$
in general (for $v \neq \Nabla u$).
\end{remark}

Let us close the subsection with additional comments on the situation of $V=\Nabla u$
with $u \in \cC^{\infty}_c(M)$.
A little knowledge of optimal transport theory is necessary
for which we refer to \cite{Vi2} and \cite{Oint}.
A potential function $\varphi_t$ of the vector field
$V_t$ (i.e., $\Nabla \varphi_t=V_t$) is given by
\begin{equation}\label{eq:HL}
\varphi_t(y):=\inf_{x \in M} \bigg\{ \frac{d(x,y)^2}{2t} +u(x) \bigg\}.
\end{equation}
This is the \emph{Hopf-Lax formula} providing
a (viscosity) solution to the \emph{Hamilton-Jacobi equation}
\begin{equation}\label{eq:HJ}
\frac{\del \varphi_t}{\del t} +\frac{F(\Nabla\varphi_t)^2}{2}=0
\end{equation}
that corresponds to \eqref{eq:Euler}.
Indeed, we can apply the proof of \cite[Theorem~2.5]{LV-HJ}
verbatim to see \eqref{eq:HJ}, although our distance is nonsymmetric.
In view of optimal transport theory, \eqref{eq:HL} is rewritten as
\[ t\varphi_t(y) =\inf_{x \in M} \bigg\{ \frac{d(x,y)^2}{2} -\big( \!-tu(x) \big) \bigg\}
 =(-tu)^c(y), \]
where $(-tu)^c$ stands for the \emph{$c$-transform} of the function $-tu$ for the cost $c(x,y)=d(x,y)^2/2$.
Thus, if we put $\eta(t)=\exp_x[t\Nabla u(x)]$, then it holds that
\[ \dot{\eta}(t) =-\overleftarrow{\Nabla}(-\varphi_t)\big( \eta(t) \big) =\Nabla\varphi_t\big( \eta(t) \big). \]
Hence, in terms of Wasserstein geometry, the vector field
$\Nabla\varphi_t$ is the tangent vector at time $t$ of the geodesic
$(\exp[t\Nabla u]_{\sharp}\mu)_{t \in [0,\delta]}$ starting from any absolutely continuous
probability measure $\mu$ (whereas $\delta$ depends on $u$),
where $\exp[t\Nabla u]_{\sharp}\mu$ denotes the push-forward of
$\mu$ by the map $\exp[t\Nabla u]$.

\subsection{Integrated formula}

We need to be careful in extending the point-wise estimates
\eqref{eq:BW}, \eqref{eq:N-BW} to global ones (in the weak sense),
because some quantities are undefined on $M \setminus M_u$.
We prove an auxiliary lemma to overcome this difficulty.

\begin{lemma}\label{lm:intBW}
Let $L$ be a locally uniformly elliptic, linear second order differential operator
in divergence form on a Riemannian manifold.
Let $\cE^L$ denote the associated Dirichlet form with domain $H^1_0(M)$
and with intrinsic gradient denoted by $\nabla^L$, i.e.,
\[ \cE^L(h,h)=-\int_M h \cdot Lh \,dm= \int_M |\nabla^L h|^2 \,dm,
 \qquad h \in H^1_0(M). \]
Then we have the following.
\begin{enumerate}[{\rm (i)}]
\item For each $h \in H^1_0(M)$, it holds $\nabla^L h=0$ a.e.\ on $h^{-1}(0)$.

\item If $h \in H^1_0(M) \cap L^{\infty}(M)$,
then $\nabla^L(h^2/2)=h\nabla^L h=0$ a.e.\ on $h^{-1}(0)$.

\item The assertions in {\rm (i)} and {\rm (ii)} also hold true if $h$ merely lies locally
in the respective spaces.
\end{enumerate}
\end{lemma}

\begin{proof}
(i) is a particular case of the general fact that, for local regular Dirichlet forms
with square field operator $\Gamma$,
\[ \cE^L(h,h)=\int_M \Gamma(h,h) \,dm =\int_{\{h \neq 0\}} \Gamma(h,h) \,dm. \]
(In the monograph~\cite{BH}, this is called the \emph{energy image density}.)

(ii) follows from the chain rule.

(iii) Given an arbitrary open relatively compact set $U \subset M$,
choose a cut-off function $\phi \in H^1_c(M) \cap L^{\infty}(M)$ such that
$\phi \equiv 1$ on $U$.
Then, for each $h \in H^1_{\loc}(M)$, it holds $h\phi \in H^1_0(M)$
and we have $\nabla^L (h\phi)=\nabla^L \phi$ a.e.\ on
$(h\phi)^{-1}(0) \cap U =h^{-1}(0) \cap U$.
Thus we can reduce the claim to (i) and (ii).
$\qedd$
\end{proof}

\begin{theorem}[Integrated Bochner-Weitzenb\"ock formula]\label{th:intBW}
Given $u \in H^2_{\loc}(M) \cap \cC^1(M)$ with $\Lap u \in H^1_{\loc}(M)$,
we have
\begin{align*}
&-\int_M D\phi\bigg( \nabla^{\Nabla u} \bigg( \frac{F(\Nabla u)^2}{2} \bigg) \bigg)
 \,dm \\
&= \int_M \phi \big\{ D(\Lap u)(\Nabla u) +\Ric_{\infty}(\Nabla u)
 +\| \Nabla^2 u \|_{HS(\Nabla u)}^2 \big\} \,dm
\end{align*}
for all $\phi \in H^1_c(M) \cap L^{\infty}(M)$, and
\begin{equation}\label{eq:N-wBW}
-\int_M D\phi\bigg( \nabla^{\Nabla u} \bigg( \frac{F(\Nabla u)^2}{2} \bigg) \bigg) \,dm
 \ge \int_M \phi \bigg\{ D(\Lap u)(\Nabla u) +\Ric_N(\Nabla u)
 +\frac{(\Lap u)^2}{N} \bigg\} \,dm
\end{equation}
for $N \in [n,\infty]$ and all nonnegative functions
$\phi \in H^1_c(M) \cap L^{\infty}(M)$.
\end{theorem}

\begin{proof}
As the proofs are common, we consider only \eqref{eq:N-wBW} with $N<\infty$.

(I) Let us first treat the case of $u \in \cC^{\infty}(M)$.
If $\phi \in H^1_c(M_u)$, then \eqref{eq:N-wBW} follows from \eqref{eq:N-BW}
via integration by parts.

For an arbitrary bounded nonnegative function $\phi \in H^1_c(M) \cap L^{\infty}(M)$,
set
\[ \phi_k:=\min\big\{ \phi,k^2 F(\Nabla u)^2 \big\}, \qquad k \in \N. \]
Then $\phi_k \in H^1_c(M_u)$ and $\lim_{k \to \infty}\phi_k(x)=\phi(x)$ for all $x \in M_u$.
Hence we have
\begin{equation}\label{eq:phi_k}
-\int_M D\phi_k \bigg( \nabla^{\Nabla u} \bigg( \frac{F(\Nabla u)^2}{2} \bigg) \bigg) \,dm
 \ge \int_M \phi_k \bigg\{ D(\Lap u)(\Nabla u) +\Ric_N(\Nabla u)
 +\frac{(\Lap u)^2}{N} \bigg\} \,dm.
\end{equation}
In the limit as $k$ goes to infinity, the right hand side converges to
\[ \int_{M_u} \phi \bigg\{ D(\Lap u)(\Nabla u) +\Ric_N(\Nabla u)
 +\frac{(\Lap u)^2}{N} \bigg\} \,dm. \]
The first two terms of the integrand obviously vanish on the set $M \setminus M_u$.
In order to see that also the third term vanishes, let us fix a local coordinate
$(x^i)_{i=1}^n$ and apply Lemma~\ref{lm:intBW} to the function
$h:=\del u/\del x^i \in H^1_{\loc}(M)$.
It implies that $\del h/\del x^j=0$ a.e.\ on $h^{-1}(0)$ for all $j=1,2,\ldots,n$.
In other words, $\del^2 u/\del x^i \del x^j=0$ for all $i,j=1,2,\ldots,n$ on
$\bigcap_{i=1}^n \{\del u/\del x^i=0\}$.
Hence, in particular, $\Lap u=0$ a.e.\ on $M \setminus M_u$.

To see the passage to the limit on the left hand side of \eqref{eq:phi_k},
we observe from Lemma~\ref{lm:intBW} that
$\nabla^{\Nabla u}(F(\Nabla u)^2)=2F(\Nabla u) \nabla^{\Nabla u}(F(\Nabla u))$
vanishes a.e.\ on $M \setminus M_u$.
Thus, putting
\[ \Omega_k:=\big\{ x \in M_u \,\big|\, \phi(x)>k^2F\big( \Nabla u(x) \big)^2 \big\}
 =\{ x \in M_u \,|\, \phi(x) \neq \phi_k(x) \}, \]
we find
\begin{align*}
&\bigg| \int_M D(\phi-\phi_k)
 \Big( \nabla^{\Nabla u} \big( F(\Nabla u)^2 \big) \Big) \,dm \bigg| \\
&\le \bigg| \int_{\Omega_k} D\phi
 \Big( \nabla^{\Nabla u} \big( F(\Nabla u)^2 \big) \Big) \,dm \bigg|
 +k^2 \int_{\Omega_k} D\big( F(\Nabla u)^2 \big)
 \Big( \nabla^{\Nabla u} \big( F(\Nabla u)^2 \big) \Big) \,dm.
\end{align*}
The first term of the right hand side tends to zero since $\Omega_k$
decreases to a null set as $k$ goes to infinity.
For the second term, note that
\[ D\big( F(\Nabla u)^2 \big)
 \Big( \nabla^{\Nabla u} \big( F(\Nabla u)^2 \big) \Big)
 = 4F(\Nabla u)^2 \cdot D\big( F(\Nabla u) \big)
 \Big( \nabla^{\Nabla u} \big( F(\Nabla u) \big) \Big) \]
on $M_u$.
Therefore, by the choice of $\Omega_k$, we obtain
\begin{align*}
k^2 \int_{\Omega_k} D\big( F(\Nabla u)^2 \big)
 \Big( \nabla^{\Nabla u} \big( F(\Nabla u)^2 \big) \Big) \,dm
&\le 4\int_{\Omega_k} \phi \cdot D\big( F(\Nabla u) \big)
 \Big( \nabla^{\Nabla u} \big( F(\Nabla u) \big) \Big) \,dm \\
&\to 0 \qquad (k \to \infty).
\end{align*}
Hence, as $k$ goes to infinity, the left hand side of \eqref{eq:phi_k} converges to
\[ -\int_M D\phi \bigg( \nabla^{\Nabla u} \bigg( \frac{F(\Nabla u)^2}{2} \bigg) \bigg) \,dm. \]
This proves the claim for $u \in \cC^{\infty}(M)$.

(II) Now let us consider the general case of $u \in H^2_{\loc}(M) \cap \cC^1(M)$
with $\Lap u \in H^1_{\loc}(M)$.
Choose $u_k \in \cC^{\infty}(M)$ with $u_k \to u$ locally in the $H^2$-norm as $k \to \infty$.
Using integration by parts, we rewrite \eqref{eq:N-wBW} for the function $u_k$ in the form
\begin{align*}
&-\int_M D\phi\bigg( \nabla^{\Nabla u_k} \bigg( \frac{F(\Nabla u_k)^2}{2} \bigg) \bigg)
 \,dm \\
&\ge -\int_M \div_m(\phi \Nabla u_k) \Lap u_k \,dm
 +\int_M \phi \bigg\{ \Ric_N(\Nabla u_k) +\frac{(\Lap u_k)^2}{N} \bigg\} \,dm.
\end{align*}
As $k$ goes to infinity, each of the terms appearing in this inequality converges towards
the repective term with $u$ in the place of $u_k$.
This finally proves the claim for general $u$.
$\qedd$
\end{proof}

\section{Applications}\label{sc:appl}

In linear semigroup theory, the estimates as in Theorems~\ref{th:pwBW}, \ref{th:intBW}
are also called the \emph{$\Gamma_2$-inequality}
or the \emph{Bakry-\'Emery curvature-dimension condition} (after \cite{BE}).
Applying it to solutions to the heat equation, we obtain Bakry-\'Emery
and Li-Yau type estimates (for $N=\infty$ and $N<\infty$, respectively).
We remark that Theorem~\ref{th:hf} ensures that solutions to the heat equation
enjoy enough regularity for applying Theorem~\ref{th:intBW}.
Throughout the section, we assume the compactness of $M$.
We will consider only global solutions $u:[0,T] \times M \lra \R$
to the heat equation (more precisely, their continuous versions on $(0,T] \times M$)
for simplicity.

Given a global solution $u:[0,T] \times M \lra \R$ to the heat equation,
let us fix a measurable one-parameter family of non-vanishing vector fields
$\{V_t\}_{t \in [0,T]}$ on $M$ such that $V_t=\Nabla u_t$ on $M_{u_t}$.
For $0 \le s<t \le T$, denote by $P^{\Nabla u}_{s,t}$
the (linear, symmetric) Markov transition operator on $L^2(M,m)$
associated with the time-dependent generator $\Delta^V$.
Precisely, for each $h \in L^2(M,m)$, $h_t=P^{\Nabla u}_{s,t} h $ for $t \in [s,T]$
is the weak solution to $\del h_t/\del t =\Delta^{V_t} h_t$ with $h_s=h$.

\begin{theorem}[Bakry-\'Emery gradient estimate]\label{th:BE}
Let $(M,F,m)$ be compact and satisfy $\Ric_{\infty} \ge K$ for some $K \in \R$,
and let $u:[0,T] \times M \lra \R$ be a global solution to the heat equation.
Then we have
\begin{equation}\label{eq:BE}
F\big( \Nabla u_t(x) \big)^2 \le e^{-2K(t-s)} P_{s,t}^{\Nabla u} \big( F(\Nabla u_s)^2 \big)(x)
\end{equation}
for all $0 \le s<t \le T$ and $x \in M$.
\end{theorem}

\begin{proof}
For fixed $t \in (0,T]$, $u$ as above and for an arbitrary nonnegative function $h \in \cC(M)$,
we set
\[ H(s):=e^{2Ks} \int_M P_{s,t}^{\Nabla u}h \cdot \frac{F(\Nabla u_s)^2}{2} \,dm,
 \qquad 0 \le s \le t. \]
Then we deduce from the definition of $P_{s,t}^{\Nabla u}$ that
\begin{align*}
H'(s) &=e^{2Ks} \int_M D\bigg( \frac{F(\Nabla u_s)^2}{2} \bigg)
 \big( \nabla^{\Nabla u_s}(P_{s,t}^{\Nabla u}h) \big) \,dm \\
&\quad +e^{2Ks} \int_M P_{s,t}^{\Nabla u}h \cdot
 \frac{\del}{\del s} \bigg( \frac{F(\Nabla u_s)^2}{2} \bigg) \,dm
 +2KH(s).
\end{align*}
Note that, thanks to Lemma~\ref{lm:intBW}, the integrals
in the right hand side are well-defined.
By \eqref{eq:f1f2}, the first term in the right hand side coincides with
\[ e^{2Ks} \int_M D(P_{s,t}^{\Nabla u}h)
 \bigg( \nabla^{\Nabla u_s} \bigg( \frac{F(\Nabla u_s)^2}{2} \bigg) \bigg) \,dm. \]
The other terms are calculated as, with the help of Theorem~\ref{th:Euler},
\begin{align*}
&e^{2Ks} \int_M P_{s,t}^{\Nabla u}h \bigg\{
 D\bigg( \frac{\del u_s}{\del s} \bigg) (\Nabla u_s) +KF(\Nabla u_s)^2 \bigg\} \,dm \\
&= e^{2Ks} \int_M P_{s,t}^{\Nabla u}h \bigg\{ D(\Lap u_s) (\Nabla u_s)
 +KF(\Nabla u_s)^2 \bigg\} \,dm.
\end{align*}
To be precise, we used the positive $0$-homogeneity of $g_{ij}$ to see, on $M_{u_s}$,
\begin{align}
\frac{\del}{\del s} \bigg( \frac{F(\Nabla u_s)^2}{2} \bigg)
&= \frac{\del}{\del s} \bigg[ \frac{1}{2} \sum_{i,j=1}^n
 g^{ij}(\Nabla u_s) \frac{\del u_s}{\del x^i} \frac{\del u_s}{\del x^j} \bigg]
 \nonumber\\
&= \sum_{i,j=1}^n g^{ij}(\Nabla u_s) \frac{\del u_s}{\del x^i}
 \frac{\del^2 u_s}{\del s \del x^j}
 +\frac{1}{2} \sum_{i,j,k=1}^n \frac{\del g^{ij}}{\del v^k}(\Nabla u_s)
 \frac{\del (\Nabla u_s)_k}{\del s} \frac{\del u_s}{\del x^i}
 \frac{\del u_s}{\del x^j} \nonumber\\
&= D\bigg( \frac{\del u_s}{\del s} \bigg)(\Nabla u_s)
 -\frac{1}{2} \sum_{i,j,k} \frac{\del g_{ij}}{\del v^k}(\Nabla u_s)
 (\Nabla u_s)_i (\Nabla u_s)_j \frac{\del (\Nabla u_s)_k}{\del s} \nonumber\\
&= D\bigg( \frac{\del u_s}{\del s} \bigg)(\Nabla u_s), \qquad
 \textrm{where}\ \Nabla u_s=\sum_{i=1}^n (\Nabla u)_i \frac{\del}{\del x^i}.
 \label{eq:BE1}
\end{align}
Therefore we have $H'(s) \le 0$ by the weak formulation of the
Bochner-Weitzenb\"ock formula
\eqref{eq:N-wBW} (for $\phi=P_{s,t}^{\Nabla u}h$ and $N=\infty$) together with
the hypothesis $\Ric_{\infty} \ge K$.

Thus $H$ is non-increasing and, in particular,
\begin{align*}
e^{2Kt} \int_M h \, \frac{F(\Nabla u_t)^2}{2} \,dm
&\le e^{2Ks} \int_M P_{s,t}^{\Nabla u}h \cdot \frac{F(\Nabla u_s)^2}{2} \,dm \\
&= e^{2Ks} \int_M h \cdot P_{s,t}^{\Nabla u} \bigg( \frac{F(\Nabla u_s)^2}{2} \bigg) \,dm.
\end{align*}
Since this holds true for any nonnegative $h$, we obtain the claimed inequality
\eqref{eq:BE} for a.e.\ $x \in M$.
The H\"older continuity in $x$ of both sides of \eqref{eq:BE} finally
allows to deduce it for every $x \in M$.
$\qedd$
\end{proof}

We stress that it is used in \eqref{eq:BE} the mixture of the nonlinear operator
$u_s \longmapsto u_t$ and the linear one $P_{s,t}^{\Nabla u}$.
As an immediate corollary to Theorem~\ref{th:BE}, we obtain a growth bound
for Lipschitz constants.
For a continuous function $u \in \cC(M)$, define
\[ \Lip(u):=\sup_{x,y \in M} \frac{f(y)-f(x)}{d(x,y)}. \]

\begin{corollary}\label{cr:BE}
Assume that $(M,F,m)$ is compact and satisfies $\Ric_{\infty} \ge K$ for some $K \in \R$,
and let $u:[0,T] \times M \lra \R$ be a global solution to the heat equation.
Then we have
\[ \| F(\Nabla u_t) \|_{L^{\infty}} \le e^{-Kt} \| F(\Nabla u_0) \|_{L^{\infty}} \]
and, if $u_0 \in \cC(M)$,
\[ \Lip(u_t) \le e^{-Kt} \Lip(u_0) \]
for all $t \in [0,T]$.
\end{corollary}

\begin{remark}\label{rm:BE}
In the Riemannian case, each of the properties in Corollary~\ref{cr:BE} characterizes
manifolds with $\Ric_{\infty} \ge K$ (\cite[Theorem~1.3]{vRS}),
and $\Ric_{\infty} \ge K$ is also equivalent to the $K$-contraction property of heat flow
in the $L^2$-Wasserstein space (\cite[Corollary~1.4]{vRS}, see \eqref{eq:intr1}).
The equivalence between the gradient estimate and the contraction property
is known to hold for more general linear semigroups by Kuwada's duality~\cite{Ku}.
Comparing Theorem~\ref{th:BE} and Corollary~\ref{cr:BE} with
the non-contraction property of heat flow established in \cite{OS2},
we observe that a completely different phenomenon occurs for nonlinear semigroups.
\end{remark}

Next we consider Li-Yau type estimates for $N<\infty$ (see \cite{LY} and \cite{Da}).

\begin{theorem}[Li-Yau gradient estimate]\label{th:LY}
Let us assume that $(M,F,m)$ is compact and satisfies $\Ric_N \ge K$
for some $N \in [n,\infty)$ and $K \in \R$,
put $K':=\min\{ K,0 \}$ and take a positive global solution
$u:[0,T] \times M \lra \R$ to the heat equation.
Then we have
\[ F\big( \Nabla (\log u)(t,x) \big)^2 -\theta \del_t(\log u)(t,x)
 \le N\theta^2 \bigg( \frac{1}{2t}-\frac{K'}{4(\theta -1)} \bigg) \]
on $(0,T] \times M$ for any $\theta>1$.
\end{theorem}

\begin{proof}
We follow the argumentation in \cite[Lemmas~5.3.2, 5.3.3]{Da} and focus
on the modifications required because of nonlinearity and
the lack of higher order regularity.
Throughout the proof, we fix a measurable one-parameter family of
non-vanishing vector fields $\{V_t\}_{t \in [0,T]}$ as described
in the paragraph preceding to Theorem~\ref{th:BE}.
We divide the proof into five steps.
Note that it suffices to show the claim for $t=T$.

(I)
First consider the function $f:=\log u$ which is $H^2$ in space and $\cC^{1,\alpha}$
in both space and time.
As $u$ solves the heat equation, $f$ satisfies
\begin{equation}\label{eq:LY1}
\Lap f_t+F(\Nabla f_t)^2=\del_t f_t
\end{equation}
for every $t$ in the weak sense that
\[ \int_M \{ -D\phi(\Nabla f_t)+\phi F(\Nabla f_t)^2 \} \,dm
 =\int_M \phi \, \del_t f_t \,dm \]
for each $\phi \in H^1(M)$.
Note also that $g_{\Nabla f_t}=g_{\Nabla u_t}$ a.e.\ on $M_{u_t}$ and
$\Lap f_t \in H^1(M)$ for each $t>0$ since $u$ is positive.

(II)
Next we verify that the function $\sigma(t,x):=t \del_t f(t,x)$
($\in H^1(M)$ by Theorem~\ref{th:hf}) satisfies
\begin{equation}\label{eq:LY2}
\Delta^V \sigma -\del_t \sigma +\frac{\sigma}{t} +2D\sigma(\Nabla f)=0
\end{equation}
in the sense of distributions on $(0,T) \times M$.
Indeed, for each $\phi \in H^1_0((0,T) \times M)$,
\begin{align*}
&\int_0^T \int_M \bigg\{ -D\phi(\nabla^V \sigma) +\del_t \phi \cdot \sigma
 +\frac{\phi \sigma}{t} +2\phi D\sigma(\Nabla f) \bigg\} \,dmdt \\
&= \int_0^T \int_M \big\{ -D(t\phi) \big( \nabla^V (\del_t f) \big)
 +\del_t(t\phi) \del_t f +2t\phi D(\del_t f)(\Nabla f) \big\} \,dmdt \\
&= \int_0^T \int_M \big\{ D\big(\del_t (t\phi) \big)(\Nabla f)
 +\del_t(t\phi) \big( \Lap f+F(\Nabla f)^2 \big) +t\phi \del_t\big( F(\Nabla f)^2 \big) \big\} \,dmdt \\
&= 0,
\end{align*}
where we used \eqref{eq:LY1} and \eqref{eq:BE1} in the second equality.
To be precise, a calculation similar to \eqref{eq:BE1} ensures
\[ \del_t \big( D(t\phi)(\Nabla f) \big)
 =D\big( \del_t(t\phi) \big)(\Nabla f) +D(t\phi) \big( \nabla^V (\del_t f) \big). \]

(III)
Now let us consider the function $\alpha(t,x):=t\{ F(\Nabla f(t,x))^2 -\theta \del_t f(t,x) \}$.
It lies in $H^1(M)$ for each $t$ and is H\"older continuous in both space and time.
Using the previous identity \eqref{eq:LY2} for $\sigma$ as well as our version of
the Bochner-Weitzenb\"ock formula, we shall deduce that
\begin{equation}\label{eq:LY3}
\Delta^V \alpha +2D\alpha(\Nabla f) -\del_t \alpha \ge \beta
\end{equation}
again in the distributional sense on $(0,T) \times M$,
where $\beta$ denotes the continuous function
\[ \beta(t,x):=-\frac{\alpha}{t} +2t \bigg\{ \frac{1}{N}\big(F(\Nabla f)^2 -\del_t f \big)^2
 +K F(\Nabla f)^2 \bigg\}. \]
Indeed, for each nonnegative $\phi \in H^1_0((0,T) \times M)$,
\eqref{eq:LY2}, \eqref{eq:BE1} and \eqref{eq:LY1} show
\begin{align*}
&\int_0^T \int_M \bigg\{ -D\phi(\nabla^V \alpha) +2\phi D\alpha(\Nabla f)
 +\del_t \phi \cdot \alpha +\frac{\phi \alpha}{t} \bigg\} \,dmdt \\
&= \int_0^T \int_M \Big\{ -tD\phi \Big( \nabla^{\Nabla u} \big( F(\Nabla f)^2 \big) \Big)
 +2t\phi D\big( F(\Nabla f)^2 \big)(\Nabla f) \\
&\qquad -\phi \big\{ F(\Nabla f)^2 +2tD(\del_t f)(\Nabla f) \big\}
 +\phi F(\Nabla f)^2 \Big\} \,dmdt \\
&= \int_0^T \int_M \Big\{ -tD\phi \Big( \nabla^{\Nabla u} \big( F(\Nabla f)^2 \big) \Big)
 -2t\phi D(\Lap f)(\Nabla f) \Big\} \,dmdt.
\end{align*}
Thanks to \eqref{eq:N-wBW} and the hypothesis $\Ric_N \ge K$, this is estimated from below by
\[ 2t \int_0^T \int_M \phi \bigg\{ K F(\Nabla f)^2 +\frac{(\Lap f)^2}{N} \bigg\} \,dmdt \]
and \eqref{eq:LY1} completes the proof of \eqref{eq:LY3}.

(IV)
Let $(s,x)$ denote the maximizer of $\alpha$ on $[0,T] \times M$ (note that $M$ is compact).
Without restriction, $\alpha(s,x)>0$ and thus $s>0$
(otherwise, the assertion of the theorem is obvious).
We shall claim that this implies $\beta(s,x) \le 0$.
Assume the contrary, $\beta(s,x)>0$.
It would imply $\beta>0$ on a neighborhood of $(s,x)$.
Hence, according to \eqref{eq:LY3} on such a neighborhood, the function $\alpha$ would be
a strict sub-solution to the linear parabolic operator
\[ \div_m(\nabla^V \alpha) +2D\alpha(\Nabla f) -\del_t \alpha. \]
Therefore, $\alpha(s,x)$ would be strictly less than the supremum of $\alpha$
on the boundary of any small parabolic cylinder $[s-\delta,s] \times B_{\delta}(x)$,
where $B_{\delta}(x):=\{ y \in M \,|\, d(x,y)<\delta \}$.
In particular, $\alpha$ could not be maximal at $(s,x)$.

(V)
The conclusion of the previous step is that $\beta(s,x) \le 0$.
In other words,
\[ \alpha \ge 2s^2 \bigg\{ \frac{1}{N}\big(F(\Nabla f)^2 -\del_t f \big)^2
 +K F(\Nabla f)^2 \bigg\} \]
at the maximum point $(s,x)$ of $\alpha$.
Following the reasoning of \cite[Lemma~5.3.3]{Da}, we conclude
\[ \alpha(s,x) \le \frac{N\theta^2}{2} \bigg( 1-\frac{K's}{2(\theta -1)} \bigg) \]
and thus, since $K' \le 0$,
\[ \alpha(T,y) \le \alpha(s,x) \le \frac{N\theta^2}{2} \bigg( 1-\frac{K'T}{2(\theta -1)} \bigg) \]
for all $y \in M$.
This completes the proof.
$\qedd$
\end{proof}

We proceed along the line of \cite[Theorem~5.3.5]{Da} and obtain
the Harnack inequality.

\begin{theorem}[Li-Yau Harnack inequality]\label{th:LYH}
Assume that $(M,F,m)$ is compact and satisfies $\Ric_N \ge K$ for some $N \in [n,\infty)$ and $K \in \R$,
put $K':=\min\{ K,0 \}$ and take a nonnegative global solution
$u:[0,T] \times M \lra \R$ to the heat equation.
Then we have
\[ u(s,x) \le u(t,y) \cdot \bigg( \frac{t}{s} \bigg)^{\theta N/2}
 \exp\bigg( \frac{\theta d(y,x)^2}{4(t-s)}-\frac{\theta K'N(t-s)}{4(\theta -1)} \bigg) \]
for any $\theta>1$, $0<s<t \le T$ and $x,y \in M$.
\end{theorem}

\begin{proof}
Replacing $u$ by $u+\ve$ if necessary, we may assume without restriction
that $u$ is positive.
Let $\eta(\tau)=\exp_y((t-\tau)v)$ for $\tau \in [s,t]$ be the reverse curve of
the minimal geodesic from $y=\eta(t)$ to $x=\eta(s)$ with suitable $v \in T_yM$.
Then obviously $F(-\dot{\eta}(\tau))=d(y,x)/(t-s)$ for all $\tau$.
We also put $f:=\log u$,
\[ \Theta :=\frac{\theta d(y,x)^2}{4(t-s)^2} -\frac{\theta K'N}{4(\theta -1)}, \qquad
 \sigma(\tau) :=f \big( \tau,\eta(\tau) \big) +\frac{\theta N}{2} \log\tau +\Theta\tau. \]
Then we have, by Theorem~\ref{th:LY},
\begin{align*}
\del_{\tau} \sigma(\tau)
&= Df(\dot{\eta}) +\del_t f +\frac{\theta N}{2\tau} +\Theta \\
&\ge -F(\Nabla f) F(-\dot{\eta}) +\frac{F(\Nabla f)^2}{\theta}
 -N\theta \bigg( \frac{1}{2t}-\frac{K'}{4(\theta -1)} \bigg) +\frac{\theta N}{2\tau} +\Theta \\
&\ge -F(\Nabla f) \frac{d(y,x)}{t-s} +\frac{F(\Nabla f)^2}{\theta} +\frac{\theta d(y,x)^2}{4(t-s)^2}
 \ge 0.
\end{align*}
Hence we obtain
\[ u(s,x) \cdot s^{\theta N/2} e^{\Theta s} =e^{\sigma(s)} \le e^{\sigma(t)}
 =u(t,y) \cdot t^{\theta N/2} e^{\Theta t} \]
which proves the claim.
$\qedd$
\end{proof}
In both Theorems~\ref{th:LY} and \ref{th:LYH}, we can choose $\theta=1$ if $K'=0$.
Thus we have the following.

\begin{corollary}\label{cr:LYH}
Let $(M,m,F)$ be compact with $\Ric_N \ge 0$ for some $N \in [n,\infty)$,
and take a nonnegative global solution $u:[0,T] \times M \lra \R$
to the heat equation.
Then we have
\[ u(s,x) \le u(t,y) \cdot \bigg( \frac{t}{s} \bigg)^{N/2}
 \exp\bigg( \frac{d(y,x)^2}{4(t-s)} \bigg) \]
for any $0<s<t \le T$ and $x,y \in M$.
\end{corollary}

\end{document}